\documentclass[11pt]{amsart}
\usepackage{wrapfig}
\usepackage[english,francais]{babel}
\usepackage{graphicx}

\usepackage{amssymb}
\usepackage{amsfonts}
\usepackage{amsmath}
\usepackage{amsthm}

\usepackage[latin1]{inputenc}
\usepackage{amsmath,amssymb,amsthm,epsfig}
\usepackage{eufrak}
\usepackage{color}
\usepackage{mathrsfs}
\usepackage{latexsym}
\usepackage{times, amscd}

\newtheorem{teo}{Theorem}

\newtheorem{cor}[teo]{Corollary}

\newcommand{\scr}{\mathscr}

\newcommand{\C}{{\mathbb{C}}}

\newcommand{\wP}{{\mathbb{P}}}
\newcommand{\umapright}[1]
{\hskip3pt\smash{\mathop{\longrightarrow}\limits^{#1}} \hskip3pt}

\begin{document}
\selectlanguage{english}
\title{Residue formula for Morita-Futaki-Bott invariant on orbifolds}
\author{ Maur\'icio Corr\^ea, Miguel Rodr\'iguez}
\address{ Maur\'icio   Corr\^ea  , Dep. Matem\'atica
ICEx - UFMG, Campus Pampulha, 31270-901 Belo Horizonte - Brasil.}
\email{mauricio@mat.ufmg.br}

\address{Miguel Rodr\'iguez ,
Dep. Matem\'atica  UFSJ,Praa Frei Orlando, 170 - Centro,  36307-352 S\~ao Jo\~ao Del Rei - MG, - Brasil} \email{miguel.rodriguez.mat@gmail.com}

\begin{abstract}
\selectlanguage{english}
In this work we prove a residue formula for the  Morita-Futaki-Bott
invariant with respect to  any holomorphic vector field, with
isolated (possibly degenerated) singularities in terms of
Grothendieck's residues.
\smallskip

\selectlanguage{francais}
\noindent{\bf R\'esum\'e.} \vskip 0.5\baselineskip \noindent
{\bf Une formule r\'esiduelle pour l'invariant de Morita-Futaki-Bott sur une orbifold } \newline
On obtient, en utilisant les r\'esidus de Grothendieck, une formule r\'esiduelle pour l'invariant de Morita-Futaki-Bott par rapport \`a un champ de vecteurs holomorphe avec singularit\'es isol\'ees, d\'eg\'en\'er\'ees ou non.

\end{abstract}
\selectlanguage{english}

\maketitle


 \footnotetext[3]{{\sl This
work was partially supported by CNPq, CAPES, FAPEMIG and FAPESP-2015/20841-5.} }

\section*{Introduction}

Let $X$ be a compact complex orbifold of dimension $n$. That is, 
$X$ is a complex space endowed with the following property:  each point $p \in X$ possesses a neighborhood which is the quotient $\widetilde{U}/G_p$, where $\widetilde{U}$ is a complex manifold, say of dimension $n$, and $G_p$ is a properly discontinuous finite group of automorphisms of $\widetilde{U}$, so that locally we have a quotient map  $(\widetilde{U}, \tilde{p}) \umapright{\pi_p} (\widetilde{U}/G_p,p)$.
 See \cite{alr}.

Let  $\eta(X)$  be  the complex Lie algebra of
all holomorphic vector fields of $X$. Choose any hermitian metric
$h$ on $X$ and let $\nabla$ and $\Theta$ be the Hermitian connection
and the curvature form with respect to $h$, respectively. Let $\xi$
be a global holomorphic vector field on $X$ and consider the
operator
$$L(\xi):=[\xi, \ \cdot  \ ]-\nabla_{\xi}(\ \cdot \ ):  T^{1,0}X \longrightarrow  T^{1,0}X  .$$

 Let $\phi$ be an invariant polynomial of
degree $n+k$; the \textit{Futaki-Morita integral invariant }  is
defined by
$$f_{\phi}(\xi)=\int_X \bar{\phi}\big(\underbrace{L(\xi),...,L(\xi)}_{k\,\,times},\underbrace{\frac{i}{2\pi}\Theta,...,\frac{i}{2\pi}\Theta}_{n\,\,times}\big),$$
where $\bar{\phi}$ denotes the polarization of $\phi$.  
Morita and
Futaki proved  in  \cite{futa} that the  definition of
$f_{\phi}(\xi)$ does not depend on the choice of  Hermitian
metric $h$.
It is well known that  the Futaki-Morita integral invariant can be
calculated via a  Bott type residue  formula for non-degenerated
holomorphic vector fields , see \cite{futaki}, \cite{futa},
\cite{futa-Mor} and \cite{tian} in the orbifold case. In this work
we prove a residue formula for   holomorphic vector fields with
isolated and possibly degenerated singularities in terms of
Grothendieck's residues(see \cite[Chapter 5]{gri}). 
\begin{teo} \label{teo1} Let $\xi\in\eta(X)$ a holomorphic vector field  with only isolated singularities, then
$${n+k\choose n}f_{\phi}(\xi)=(-1)^k \sum_{p\, \in Sing(\xi)}\frac{1}{\# G_p}\mbox{Res}_{\,\tilde{p}}\left\{\frac{\phi\big(J\tilde{\xi}\,\big)\,d\tilde{z}_1\wedge\dots\wedge d\tilde{z}_n}{\tilde{\xi}_1\dots\tilde{\xi}_n}\right\},$$
where, given $p$ such that $\xi(p)=0$ and $(\widetilde{U},
\tilde{p}) \umapright{\pi_p} (\widetilde{U}/G_p,p)$ denotes the
projection: $\tilde{\xi}= \pi_p^\ast \xi$,
$J{\tilde{\xi}}=\displaystyle\left(\frac{\partial\tilde{\xi}_i}{\partial
\tilde{z}_j}\right)_{1\leq i,\,j\leq n}$ and
$\mbox{Res}_{\,\tilde{p}}\displaystyle\left\{\frac{\phi\big(J\tilde{\xi}\,\big)\,d\tilde{z}_1\wedge\dots\wedge
d\tilde{z}_n}{\tilde{\xi}_1\dots\tilde{\xi}_n}\right\}$ is
Grothendieck's point residue and $(\tilde{z}_1,\dots,\tilde{z}_n)$ is a germ of  coordinate system on 
$(\widetilde{U},\tilde{p})$.
\end{teo}

We note that such residue can be calculated using Hilbert's
Nullstellensatz and Martinelli's integral formula. In fact, if
$\tilde{z}_i^{\,a_i}=\sum_{j=1}^n b_{ij} \, \tilde{\xi}_j$, then
(see \cite{norg})
\begin{equation}\label{ident}
\mbox{Res}_{\,\tilde{p}}\left\{\frac{\phi\big(J\tilde{\xi}\,\big)\,d\tilde{z}_1\wedge\dots\wedge
d\tilde{z}_n}{\tilde{\xi}_1\dots\tilde{\xi}_n}\right\} =
\frac{1}{\prod_{i=1}^n (a_{i}-1)!} \left(
\frac{\partial^{\,n}}{\partial
\tilde{z}_1^{\,a_1},...,\tilde{z}_n^{\,a_n}}
\big(Det(b_{ij})\phi(J\tilde{\xi})\big)\right)(\tilde{p}).
\end{equation}
\noindent Moreover, note that if $p\in Sing(\xi)$ is a not
degenerated  singularity  of $\xi$  then
$$\mbox{Res}_{\,\tilde{p}}\left\{\frac{\phi\big(J\tilde{\xi}\,\big)\,d\tilde{z}_1\wedge\dots\wedge d\tilde{z}_n}{\tilde{\xi}_1\dots\tilde{\xi}_n}\right\} = \frac{\phi\big(J\tilde{\xi}(\tilde{p})\big)}{Det\big(J\tilde{\xi}(\tilde{p})\big)}.$$
Theorem $1$ allows us  to calculate the Morita-Futaki invariant for holomorphic 
vector fields with possible degenerated singularities. For instance,
in recent  work   \cite{li}  the Futaki-Bott residue for vector fields with
degenerated singularities, on the blowup of K\"ahler surfaces, was
calculated by Li and Shi. Compare the equation (\ref{ident}) with
Lemma 3.6 of  \cite{li}.

Futaki showed in  \cite{futaki} that if $X$ admits  a K\"ahler-Einstein metric then
$f_{C_{1}^{\,n+1}}\equiv 0$, where  $C_1=Tr$ denotes  the trace, i.e., the first  elementary symmetric polynomial. Taking  $\phi=C_1^{\,n+1}=Tr^{\,n+1}$, we obtain 
the following corollary of  Theorem \ref{teo1}.
\begin{cor} \label{cor3} Let $\xi\in\eta(X)$ with only isolated singularities, then
$$f_{C_{1}^{\,n+1}}(\xi)=\frac{-1}{(n+1)^2} \sum_{p\, \in Sing(\xi)}\frac{1}{\# G_p}\mbox{Res}_{\,\tilde{p}}\left\{\frac{Tr^{\,n+1}\big(J\tilde{\xi}\,\big)\,d\tilde{z}_1\wedge\dots\wedge d\tilde{z}_n}{\tilde{\xi}_1\dots\tilde{\xi}_n}\right\}.$$
\end{cor}
This result generalizes    the Proposition $1.2$ of \cite{tian}. It
is well known that projective planes  are K\"ahler-Einstein. However,
the non-existence of   K\"ahler-Einstein metrics on singular weighted
projective planes    was shown in previous works, see for example \cite{V}. As an application of Theorem \ref{teo1}
we will give, in Section \ref{WPS}, a new proof 
of this fact.

\section{  Non-existence of K\"ahler-Einstein metric on  weighted  projective planes }\label{WPS}

Here we consider weighted complex projective planes with only
isolated singularities, which we briefly recall.

Let $w_{0},w_{1},w_{2}$ be positive integers two by two co-primes, set
$w:=(w_{0},w_{1},w_{2})$ and $|w|:=w_{0}+w_{1}+w_{2}$. Define an
action of $\C^\ast$ in $\C^{3} \setminus \{0\}$ by
$$
\begin{array}{ccc}
\C^\ast \times \C^{3} \setminus \{0\} & \longrightarrow & \C^{3} \setminus \{0\} \\
\lambda . (z_0, z_1, z_2) & \longmapsto &  (\lambda^{w_0} z_0, \lambda^{w_1} z_1, \lambda^{w_2} z_2)\\
\end{array}
$$
and let $\wP_w^2 := \C^{3} \setminus \{0\} / \sim$. The weights
are chosen to be pairwise co-primes in order to assure a finite number
of singularities and to give $\wP_w^2$ the structure of an
effective, abelian, compact orbifold of dimension $2$.  The singular
locus is:
$$
Sing (\wP_w^2) = \big\{[1:0:0]_{\omega}, [0:1:0]_{\omega},[0:0:1]_{\omega}\big\}.
$$
We have the canonical projection
$$
\begin{array}{ccc}
\pi: {\C}^{3} \setminus \{0\} & \longrightarrow& \wP_w^2 \hfill \\
(z_0, z_1, z_2)& \longmapsto & [z_0^{w_0}: z_1^{w_1} : z_2^{w_2}]_w
\end{array}
$$
and the natural map
$$
\begin{array}{ccc}
\hskip 32pt \varphi_w: \wP^n & \longrightarrow& \wP_w^n \hfill \\
{ [z_0: z_1 : z_2]}  & \longmapsto &   [z_0^{w_0}: z_1^{w_1} : z_2^{w_2}]_w
\end{array}
$$
of degree $\deg \varphi_w= w_0 w_1 w_2$. The map $\varphi_w$ is
\emph{good} in the sense of \cite[section 4.4]{alr}, which
means, among other things, that V-bundles behave well under
pullback. It is shown in \cite{man} that there is a line V-bundle
$\scr{O}_{\wP^2_w} (1)$ on $\wP^2_w$, unique up to isomorphism, such
that
$$
\varphi_w^\ast \scr{O}_{\wP_w^2} (1) \cong  \scr{O}_{\wP^2} (1)
$$
and, by naturality, $c_1(\varphi_w^\ast \scr{O}_{\wP_w^2} (1)) = c_1
( \scr{O}_{\wP^2} (1))= \varphi_w^\ast c_1(\scr{O}_{\wP_w^2} (1))$,
from which we obtain the Chern number
$$
[\wP_w^2] \frown \left( c_1( \scr{O}_{\wP_w^2} (1))\right)^n =
\int\limits_{\wP_w^n} \left( c_1( \scr{O}_{\wP_w^2} (1))\right)^2 =
\dfrac{1}{w_0 w_1 w_2}
$$
since
$$
1 = \int\limits_{\wP^2} \left(c_1(\scr{O}_{\wP^2} (1))\right)^2 =
\int\limits_{\wP^2} \varphi_w^\ast\left( c_1( \scr{O}_{\wP_w^2}
(1))\right)^2 = (\deg \varphi_w) \int\limits_{\wP_w^2} \left( c_1(
\scr{O}_{\wP_w^2} (1))\right)^2.
$$
There exist an  Euler type  sequence on $\wP^n_w$
$$
0\longrightarrow \underline{\mathbb{C}} \longrightarrow
\bigoplus_{i=0}^{2} \scr{O}_{\wP_w^2}(w_i)\longrightarrow
T\mathbb{P}_{w}^{2}\longrightarrow 0,
$$
where
\begin{itemize}
  \item[(i)] $1\longmapsto ({w_0}{z_0},{w_1}{z_1},{w_2}{z_2})$.
  \item[(ii)] $(P_{0},P_{1},P_{2})\longmapsto \pi_{\ast}\left(\sum_{i=0}^{2}P_{i}\frac{\partial}{\partial z_{i}}\right)$.
\end{itemize}

It is well known that the non-singular   weighted projective planes 
admit  K\"ahler-Einstein metrics.  On the other side, singular
weighted projective spaces do not admit K\"ahler-Einstein
metrics, see  \cite{V}. We give  a
simple proof   of the non-existence of   K\"ahler-Einstein metrics  on
singular $\mathbb{P}_{\omega}^{2}$  by using   Corollary
\ref{cor3}.

\begin{teo}
\emph{The singular weighted projective space
$\mathbb{P}_{\omega}^{2}$   does not admit any
K\"ahler-Einstein metric}
\end{teo}
\begin{proof}
Choose $a_{0}, a_{1},a_{2}\in \mathbb{C}^*$ such that  $a_{i}w_{j}\neq
a_{j}w_{i}$, for all  $ i\neq j$.
Suppose, without loss of generality,  that $1 \leq w_0 \leq w_2 < w_1$.
Consider the holomorphic vector field on $\mathbb{P}_{\omega}^{2}$ given 
by
$$\xi_a=\sum_{k=0}^{2}a_{k}Z_{k}\frac{\partial}{\partial Z_{k}} \in H^{0}(\mathbb{P}^{2}_{\omega},T\mathbb{P}^{2}_{\omega}),$$
where $(Z_0,Z_1 ,Z_3)$ denotes   the homogeneous coordinate system. 

\noindent The local expression of $\xi$ over $U_i=\{ [Z_0:Z_1 :Z_3] \in \mathbb{P}^2; \ Z_i \neq 0  \}$ is given by
$${\xi_a}|_{U_i}=\sum_{\stackrel{k=0}{k\neq i}}^{2}\left(a_{k}-a_{i}\frac{w_{k}}{w_{i}}\right)Z_{k}\frac{\partial}{\partial Z_{k}},$$
\noindent Therefore,  the singular set $Sing({\xi}|_{U_i})$ is reduced to $\{0\}$ and it is nondegenerate.
In general
$$Sing(\xi_a)= \big\{[1:0:0]_{\omega}, [0:1:0]_{\omega},[0:0:1]_{\omega}\big\}=Sing(\mathbb{P}^{2}_{\omega}).$$
\noindent It follows from Corollary  \ref{cor3}  that 
$$f(\xi_a) = \frac{-1}{3^2} \sum_{i=0}^{2} \frac{1}{w_{i}^{2}} \frac{\big(\sum_{k\neq i} (a_{k}w_{i}-a_{i}w_{k})\big)^{3}} {\prod_{k\neq i}(a_{k}w_{i}-a_{i}w_{k})}.$$
Thus
$$\zeta(a_0,a_1,a_2)=-3^2 w_0^2 w_1^2 w_2^2 \prod_{0\leq i < j \leq 2}(a_{i}w_{j}-a_{j}w_{i})f(\xi_a)=
$$
\noindent $
(3 w_1^5 w_2^2 w_0 -3 w_1^4 w_2^3 w_0 +3 w_1^3 w_2^4 w_0  +3 w_1^2 w_2^5 w_0 -3 w_0^4 w_2^2 w_1^2 +3 w_0^3 w_2^3 w_1^2+
6 w_0^2 w_2^4 w_1^2+$\\

\noindent $ +3 w_0^4 w_1^2 w_2^2 -3 w_0^3 w_1^3 w_2^2 -6 w_0^2 w_1^4 w_2^2)  \cdot a_1 a_2 a_0^2+\cdots$
\\
\\
\noindent is a homogeneous polynomial of degree $4$ in the variables $a_0,a_1,a_2$.
Suppose by contradiction that $\zeta(a_0,a_1,a_2)\equiv 0$. In particular the coefficient of the monomial $a_0^2 a_1 a_2$ is  zero. Thus, we have the following equation
$$w_2(w_1w_2+w_2^2+w_0^2+2w_0w_2)=w_1(w_1w_2+w_1^2+w_0^2+2w_0w_1).$$
This contradicts  $1 \leq w_0 \leq w_2 < w_1$. Thus the non-vanishing of  $\zeta(a_0,a_1,a_2)$ implies that  $f(\xi_a)$ is not zero.
Therefore  $\mathbb{P}_{\omega}^{2}$   does
not admit   K\"ahler-Einstein metrics.

\end{proof}
\section{Proof of Theorem 1}
For the proof we will use   Bott-Chern's transgression method, see
\cite{bott} and \cite{ch}.

Let $p_1,\dots,p_m$ be the zeros of $\xi$. Let $\{U_\beta\}$ be an
open cover orbifold of $X$ $(\,\varphi_\beta:\widetilde{U}_{\beta}
\rightarrow U_\beta \subset X \mbox{\,\,coordinate map}\,)$. Suppose
that $\{U_\beta\}$ is a trivializing neighborhood for the holomorphic tangent orbibundle   $TX$(see \cite[section 2.3]{alr})  of $X$ and that
we have disjoint neighborhoods coordinates $U_\alpha$ with $\,\,
p_{\alpha}\in U_{\alpha}$ and  $p_{\alpha}\not\in U_{\beta}$ if
$\alpha\neq \beta$ . On each ${\widetilde{U}_\alpha}$, take local
coordinates
$\tilde{z}^{\alpha}=(\tilde{z}^{\alpha}_1,\dots,\tilde{z}^{\alpha}_n)$
and the  holomorphic frame  $\{\frac{\partial}{\partial
\tilde{z}_1^\alpha},\dots,\frac{\partial}{\partial
\tilde{z}_n^\alpha}\}$ of $TX$. Thus, we have a local representation
$${\tilde{\xi}}^\alpha=\sum{\tilde{\xi}}^\alpha_i \frac{\partial}{\partial \tilde{z}_i^\alpha},$$

\noindent where ${\tilde{\xi}}^\alpha_i$ are holomorphic functions
in $\widetilde{U}_\alpha$, $1\leq i\leq n$.
Let $\tilde{h}_\alpha$ the Hermitian metric in
$\widetilde{U}_{\alpha}$ defined by $\langle\partial/\partial
\tilde{z}_{i}^{\alpha}, \partial/\partial
\tilde{z}_{j}^{\alpha}\rangle=\delta^{\,i}_{j}$. Also consider
$\widetilde{U}'_{\alpha}\subset \widetilde{U}_{\alpha}$ e
${U}'_{\alpha}=\varphi_\alpha(\widetilde{U}'_{\alpha})$ for each
$\alpha$. Take a Hermitian metric $h_0$ in any
$X\backslash\cup_{\alpha}\{p_\alpha\}$ and $\{\rho_{0},
\rho_{\alpha}\}$ a partition of unity subordinate to the cover
$\{X\backslash\cup_{\alpha} \overline{U'_{\alpha}}\,, {U}_{\alpha}
\}_{\alpha}$. Define a Hermitian metric $h=\rho_{0}h_{0}+\sum
\rho_{\alpha}h_{\alpha}$ in $X$. Then we have that for every
$\alpha$, the metric curvature
$\Theta\equiv0$ in $U'_{\alpha}$.

\noindent Consider  the matrix of the metric connection $\nabla$ in the open $\widetilde{U}^{\beta}$ given by 
$\theta^{\beta} =(\sum_{k}\Gamma^{\beta j}_{ik}d\tilde{z}_{k}^{\beta}$).
\noindent The local expression of $L(\xi)$   is given by
$\tilde{E}^{\beta}=(\tilde{E}^{\beta}_{ij})$ such that
$$\tilde{E}^{\beta}_{ij}=-\frac{\partial\tilde{\xi}_{i}^{\beta}}{\partial \tilde{z}_{j}^{\beta}}-\sum_{s}\Gamma^{\beta i}_{js}\tilde{\xi}^{\beta}_{s},$$
\noindent see \cite{bott} and  \cite{gri}. We indicate by $\mathcal{A}^{p,\,q}(X)$
the vector space of complex-valued $(p+q)$-forms on $X$ of type $(p,q)$. 
Define
$$\phi_{r}:={n+k\choose r}\bar{\phi}(\underbrace{E,...,E}_{n+k-r},\underbrace{\Theta,...,\Theta}_{r})\in
\mathcal{A}^{r,\,r}(X)\,\,\,\, r=0,...,n.$$
\noindent Let $\omega\in\mathcal{A}^{1,0}(X)$ in $X\backslash
Sing(\xi)$, with $\omega(\xi)=1$. Following the  Bott's idea (see \cite{bott}), it is
sufficient to show that there exists $\psi$ such that
$i(\xi)(\bar{\partial}\,\psi+\phi_{n})=0$
on $X\backslash Sing(\xi)$. We take
$\psi=\sum_{r=0}^{n-1}\psi_{r}$ such that
$$\psi_{r}=\omega\wedge(\bar{\partial}\omega)^{n-r-1}\wedge\phi_{r} \in \mathcal{A}^{n,\,n-1}(X)\ \ \ \ r=0,...,n-1.$$
\noindent The following formulas hold (see \cite{bott} or
\cite{gri}) : \vspace{0.3cm}
\begin{enumerate}
    \item [a)] $\bar{\partial}\,\Theta=0$, $\bar{\partial}E=i(\xi)\Theta$
    \item [b)] $\bar{\partial}\,\phi_{r}=i(\xi)\phi_{r+1}\,\,\,\, r=0,...,n+1$
    \item [c)] $i(\xi)\bar{\partial}\omega=0$
\end{enumerate}
Le us prove   $b)$ : since $\bar{\partial}\,\Theta=0$ and
$\bar{\partial}E=i(\xi)\Theta$, we have
$$\bar{\partial}\,\phi_{r}={n+k\choose r}\sum_{i=1}^{n+k-r}\bar{\phi}(E,...,i(\xi)\Theta,...,E,\Theta,...,\Theta)
=i(\xi)\phi_{r+1}.$$
\noindent Therefore, $a)$, $b)$ and $c)$ implies that  on
$X\backslash Sing(\xi)$ we get
$$i(\xi)(\bar{\partial}\,\psi+\phi_{n})=0.$$
Therefore $d\psi=\bar{\partial}\psi=-\phi_{n}$ on  $X\backslash
Sing(\xi)$. Thus, by   Satake-Stokes Theorem  we have
\begin{eqnarray}
{n+k\choose n}f_{\phi}(\xi)&=&\left(\frac{i}{2\pi}\right)^n \int_X
\phi_n=\left(\frac{i}{2\pi}\right)^n\lim_{\epsilon\rightarrow\,0}
\int_{X\backslash \cup_{\alpha}B_{\epsilon}(p_{\alpha})} \phi_n \nonumber\\
&=&-\left(\frac{i}{2\pi}\right)^n\lim_{\epsilon\rightarrow\,0}
\int_{X\backslash \cup_{\alpha}B_{\epsilon}(p_{\alpha})} d\psi
=\left(\frac{i}{2\pi}\right)^n\lim_{\epsilon\rightarrow\,0}
\sum_{\alpha}\int_{\partial B_{\epsilon}(p_{\alpha})} \psi^{\alpha},
\label{eq:eq6}
\end{eqnarray}
 where is $ B_{\epsilon}(p_{\alpha})=B_{\epsilon}(\tilde{p}_{\alpha})/G_{p_\alpha}$ and $B_{\epsilon}(\tilde{p}_{\alpha})$ is 
an Euclidean ball  centered at $\tilde{p}_{\alpha}$ such that $\overline{B_{\epsilon}(\tilde{p}_{\alpha})} \subset U'_{\alpha}$.
Since our metric is Euclidean in
$B_{\epsilon}(\tilde{p}_{\alpha})$, its   connection is zero and
$$\tilde{E}^{\alpha}_{ij}=-\frac{\partial\tilde{\xi}_{i}^{\alpha}}{\partial
\tilde{z}_{j}^{\alpha}}.$$

Now, by our choice of metric, $\Theta$ and hence   $\phi_{r}$, for
$r>0$, vanishes identically in $B_{\epsilon}(\tilde{p}_{\alpha})$.
Then, we have
$$\tilde{\psi}^{\alpha}=\tilde{\psi}_{0}^{\alpha}={\omega \wedge (\overline{\partial} \,\omega)}^{n-1}\phi(\tilde{E}^{\alpha})= (-1)^{n+k}{\omega \wedge (\overline{\partial} \,\omega)}^{n-1}\phi(J\tilde{\xi}^{\,\alpha})$$ on  $B_{\epsilon}(\tilde{p}_{\alpha}).$
Therefore
\begin{eqnarray}\label{eq:eq7}
\tilde{\psi}^{\alpha}=(-1)^k{\omega \wedge (\overline{\partial}
\,\omega)}^{n-1}\phi(J\tilde{\xi}^{\alpha}). \label{eq:eq2}
\end{eqnarray}\\
\noindent Consider the map $\Phi:{\mathbb{C}}^n \rightarrow
{\mathbb{C}}^{2n}$ given by
$\Phi(\tilde{z})=(\tilde{z}+\tilde{\xi}(\tilde{z}),\tilde{z})$.
There is a $(2n, 2n-1)$ closed form $\beta_n$ in ${\mathbb{C}}^{2n}
\backslash \{0\}$ (the Bochner-Martinelli kernel) such that
\begin{eqnarray}\label{eq:eq8}
\Phi^{\ast}\beta_n =\left(\frac{i}{2\pi}\right)^n \, {\omega \wedge
(\overline{\partial} \,\omega)}^{n-1}.
\end{eqnarray}
 Finally, if we substitute  (\ref{eq:eq7}) and (\ref{eq:eq8}) in (\ref{eq:eq6}),  and by using   Martinelli's formula ( \cite[pg. 655]{gri})
$$
\int_{\partial B_{\epsilon}(\tilde{p}_{\alpha})}
\phi(J\tilde{\xi}^{\,\alpha})\,\Phi^{\ast}\beta_n=
\mbox{Res}_{\,\tilde{p}_{\alpha}}\left\{\frac{\phi\big(J\tilde{\xi}^{\,\alpha}\big)\,d\tilde{z}_1\wedge\dots\wedge
d\tilde{z}_n}{\tilde{\xi}_1\dots\tilde{\xi}_n}\right\}
$$
we obtain
\begin{eqnarray*}
{n+k\choose
n}f_{\phi}(\xi)&=&
 (-1)^k \sum_{\alpha} \frac{1}{\# G_{p_{\alpha}}} \mbox{Res}_{\,\tilde{p}_{\alpha}}\left\{\frac{\phi\big(J\tilde{\xi}^{\,\alpha}\big)\,d\tilde{z}_1\wedge\dots\wedge d\tilde{z}_n}{\tilde{\xi}_1\dots\tilde{\xi}_n}\right\}.\\
\end{eqnarray*}
\bigskip
\noindent\textsc{Acknowledgements}.
We are grateful to   Marcio Soares and Marcos Jardim for many stimulating
conversations on this subject.


\end{document}